\documentclass[12pt,reqno]{amsart}
\usepackage{amssymb, a4wide}
\usepackage[all]{xy}
\usepackage{amsmath,amsthm,amscd}
\usepackage{tikz-cd}
\usepackage{enumitem}
\usepackage{mathtools}
\usepackage{graphicx}
\usepackage{bbold}

\newtheorem{theorem}{Theorem}[section]
\newtheorem{proposition}[theorem]{Proposition}
\newtheorem{lemma}[theorem]{Lemma}

\theoremstyle{definition}
\newtheorem{definition}[theorem]{Definition}
\newtheorem{example}[theorem]{Example}

\theoremstyle{remark}
\newtheorem{remark}[theorem]{Remark}

\numberwithin{equation}{section}

\def\al{\alpha}
\def\be{\beta}

\def\vp{\varphi}

\def\ol{\overline}

\def\k{\operatorname{\mathbf{k}}}
\def\Ann{\operatorname{Ann}}

\def\Bider{\operatorname{Bider}}
\def\st{\operatorname{st}}
\def\ad{\operatorname{ad}}
\def\Ad{\operatorname{Ad}}
\def\Z{\operatorname{Z}}

\def\Inn{\operatorname{InnBider}}
\def\Out{\operatorname{OutBider}}

\def\Im{\operatorname{Im}}
\def\Ker{\operatorname{Ker}}

\def\Lb{\operatorname{\textbf{\textsf{Lb}}}}
\def\XLb{\operatorname{\textbf{\textsf{XLb}}}}

\def\Der{\operatorname{Der}}

\def\id{\operatorname{id}}

\newcommand{\lb}[1]{\mathfrak{#1}}
\newcommand{\lie}[1]{\mathbb{#1}}

\def\Act{\operatorname{Act}}

\begin{document}

\title[Actor of a crossed module of Leibniz algebras]{Actor of a crossed module of Leibniz algebras}

\author[J. M. Casas]{Jos\'e Manuel Casas}
\address{Department of Applied Mathematics I, University of Vigo, E. E. Forestal, 36005	Pontevedra, Spain}
\email{jmcasas@uvigo.es}
\author[R. F. Casado]{Rafael ~F.~Casado}
\address{Department of Algebra, Universidade de Santiago de Compostela\\
15782 Santiago de Compostela, Spain}
\email{rapha.fdez@gmail.com}
\author[X. Garc\'ia-Mart\'inez]{Xabier Garc\'ia-Mart\'inez}
\address{Department of Algebra, Universidade de Santiago de Compostela\\
15782 Santiago de Compostela, Spain}
\email{xabier.garcia@usc.es}
\author[E. Khmaladze]{Emzar Khmaladze}
\address{A. Razmadze Mathematical Institute, Tbilisi State University,
	Tamarashvili St. 6, 0177 Tbilisi, Georgia}
\email{e.khmal@gmail.com}

\thanks{The authors were supported by Ministerio de Econom\'ia y Competitividad (Spain), grant MTM2013-43687-P (European FEDER support included).
 The third author was also supported by Xunta de Galicia, grant GRC2013-045 (European FEDER support included) and by an FPU scholarship, Ministerio de Educaci\'on, Cultura y Deporte (Spain).}

\begin{abstract}
We extend to the category of crossed modules of Leibniz algebras the notion of biderivation via the action of a Leibniz algebra. This results into a pair of Leibniz algebras which allow us to construct an object which is the actor under certain circumstances. Additionally, we give a description of an action in the category of crossed modules of Leibniz algebras in terms of equations. Finally, we check that, under the aforementioned conditions, the kernel of the canonical map from a crossed module to its actor coincides with the center and we introduce the notions of crossed module of inner and outer biderivations.
\end{abstract}
\maketitle

\section{Introduction}
In the category of groups it is possible to describe an action via an object called the actor, which is given by the group of automorphisms. Its analogue in the category of Lie algebras is the Lie algebra of derivations. Groups and Lie algebras are examples of categories of interest, introduced by Orzech in \cite{Or}. For these categories (see \cite{Mo} for more examples), Casas, Datuashvili and Ladra \cite{CaDaLa} gave a procedure to construct an object that, under certain circumstances, plays the role of actor. For the particular case of Leibniz algebras (resp. associative algebras) that object is the Leibniz algebra of biderivations  (resp. the algebra of bimultipliers).

In \cite{No}, Norrie extended the definition of actor to the $2$-dimensional case by giving a description of the corresponding object in the category of crossed modules of groups. The analogue construction for the category of crossed modules of Lie algebras is given in \cite{CaLa}. Regarding the category of crossed modules of Leibniz algebras, it is not a category of interest, but it is equivalent to the category of $cat^1$-Leibniz algebras (see for example \cite{CaKhLa}), which is itself a modified category of interest in the sense of \cite{BoCaDaUs}. Therefore it makes sense to study representability of actions in such category under the context of modified categories of interest, as it is done in \cite{BoCaDaUs} for crossed modules of associative algebras.

Bearing in mind the ease of the generalization of the actor in the category of groups and Lie algebras to crossed modules, together with the role of the Leibniz algebra of biderivations, it makes sense to assume that the analogous object in the category of crossed modules of Leibniz algebras will be the actor only under certain hypotheses. In \cite{CaInKhLa} the authors gave an equivalent description of an action of a crossed module of groups in terms of equations. A similar description is done for an action of a crossed module of Lie algebras (see \cite{CaCaKhLa}). In order to extend the notion of actor to crossed modules of Leibniz algebras, we generalize the concept of biderivation to the $2$-dimensional case, describe an action in that category in terms of equations and give sufficient conditions for the described object to be the actor.

The article is organized as follows: In Section~\ref{section2} we recall some basic definitions on actions and crossed modules of Leibniz algebras. In Section~\ref{section3} we construct an object that extends the Leibniz algebra of biderivations to the category of crossed modules of Leibniz algebras (Theorem~\ref{theo_action_bider_nqmu}) and give a description of an action in such category in terms of equations. In Section~\ref{section4} we find sufficient conditions for the previous object to be the actor of a given crossed module of Leibniz algebras (Theorem~\ref{theo_equiv_XLb_action}). Finally, in Section~\ref{section5} we prove that the kernel of the canonical homomorphism from a crossed module of Leibniz algebras to its actor coincides with the center of the given crossed module. Additionally, we introduce the notions of crossed module of inner and outer biderivations and show that, given a short exact sequence in the category of crossed modules of Leibniz algebras, it can be extended to a commutative diagram including the actor and the inner and outer biderivations.

\section{Preliminaries}\label{section2}
In this section we recall some needed basic definitions. Throughout the paper we fix a commutative ring with unit $\k$. All algebras are considerer over $\k$. 
\begin{definition}[\cite{Lo_Lbalg}]
	A \emph{Leibniz algebra} $\lb{p}$ is a $\k$-module together with a bilinear operation $[ \ , \ ]\colon\lb{p}\times\lb{p}\to\lb{p}$, called the Leibniz bracket, which satisfies the Leibniz identity:
	\[
	[[p_1,p_2],p_3]=[p_1,[p_2,p_3]]+[[p_1,p_3],p_2],
	\]
	\noindent for all $p_1,p_2,p_3\in\lb{p}$.
	
	A homomorphism of Leibniz algebras is a $\k$-linear map that preserves the bracket.
\end{definition}
We denote by $\Ann(\lb{p})$ (resp. $[\lb{p},\lb{p}]$) the \emph{annihilator} (resp. commutator) of $\lb{p}$, that is the subspace of $\lb{p}$ generated by
\begin{equation*}
\{p_1 \in \lb{p} \ | \ [p_1,p_2] = [p_2,p_1] = 0, \ \text{for all } p_2 \in \lb{p}\}
\end{equation*}
\begin{equation*}
(\text{resp. }\{[p_1,p_2] \ | \ \text{for all } p_1, p_2 \in \lb{p}\})
\end{equation*}
It is obvious that both $\Ann(\lb{p})$ and $[\lb{p},\lb{p}]$ are ideals of $\lb{p}$.
\begin{definition}[\cite{LoPi}]\label{def_action_lb}
	Let $\lb{p}$ and $\lb{m}$ be two Leibniz algebras. An \emph{action} of $\lb{p}$ on $\lb{m}$ consists of a pair of bilinear maps, $\lb{p}\times \lb{m}\to \lb{m}$, $(p,m)\mapsto \left[p,m\right]$ and $\lb{m}\times \lb{p}\to \lb{m}$, $(m,p)\mapsto \left[m,p\right]$, such that
	\begin{align*}
		[p,[m,m']] & = [[p,m],m']-[[p,m'],m], \\	
		[m,[p,m']] & = [[m,p],m']-[[m,m'],p], \\
		[m,[m',p]] & = [[m,m'],p]-[[m,p],m'], \\	
		[m,[p,p']] & = [[m,p],p']-[[m,p'],p], \\
		[p,[m,p']] & = [[p,m],p']-[[p,p'],m], \\
		[p,[p',m]] & = [[p,p'],m]-[[p,m],p'],
	\end{align*}
	\noindent for all $m, m' \in \lb{m}$ and $p, p'\in \lb{p}$.
\end{definition}

Given an action of a Leibniz algebra $\lb{p}$ on $\lb{m}$, we can consider the \emph{semidirect product} Leibniz algebra
$\lb{m}\rtimes \lb{p}$, which consists of the $\k$-module $\lb{m}\oplus \lb{p}$ together with the Leibniz bracket given by
\[
[(m,p),(m',p')]=([m,m']+[p,m']+[m,p'], [p,p']),
\]
\noindent for all $(m,p), (m',p')\in \lb{m}\oplus \lb{p}$.
\begin{definition}[\cite{LoPi}]
	A \emph{crossed module of Leibniz algebras} (or Leibniz crossed module, for short) $(\lb{m}, \lb{p}, \eta)$ is a
	homomorphism of Leibniz algebras $\eta \colon \lb{m}\to \lb{p}$ together with an action of $\lb{p}$ on $\lb{m}$ such that
	\begin{align}
	& \eta([p,m])=[p,\eta(m)] \quad \text{and} \quad \eta([m,p])=[\eta(m),p],\label{XLb1} \tag{XLb1}\\
	& [\eta(m),m']=[m,m']=[m,\eta (m')],\label{XLb2}\tag{XLb2}
	\end{align}
	\noindent for all $m,m' \in \lb{m}$, $p \in \lb{p}$.
	
	A \emph{homomorphism of Leibniz crossed modules} $(\varphi, \psi)$ from $(\lb{m},\lb{p},\eta)$ to $(\lb{n},\lb{q}, \mu)$
	is a pair of Leibniz homomorphisms, $\varphi \colon \lb{m} \to \lb{n}$ and $\psi \colon \lb{p} \to \lb{q}$, such that they commute with $\eta$ and $\mu$ and they respect the actions, that is $\varphi([p,m]) = [\psi(p),\varphi(m)]$ and $\varphi([m,p]) = [\varphi(m),\psi(p)]$ for all $m\in \lb{m}$, $p\in \lb{p}$.
\end{definition}
Identity \eqref{XLb1} will be called \emph{equivariance} and \eqref{XLb2} \emph{Peiffer identity}. We will denote by $\XLb$ the category of Leibniz crossed modules and homomorphisms of Leibniz crossed modules.

Since our aim is to construct a $2$-dimensional generalization of the actor in the category of Leibniz algebras, let us first recall the following definitions.
\begin{definition}[\cite{Lo_Lbalg}]
	Let $\lb{m}$ be a Leibniz algebra. A \emph{biderivation} of $\lb{m}$ is a pair $(d,D)$ of $\k$-linear maps $d,D \colon \lb{m} \to \lb{m}$ such that
	\begin{align}
	d([m,m']) & = [d(m),m'] + [m,d(m')],\label{eq_bider_1}\\
	D([m,m']) & = [D(m),m'] - [D(m'),m],\label{eq_bider_2}\\
	[m,d(m')] & = [m, D(m')],\label{eq_bider_3}
	\end{align}
	\noindent for all $m,m' \in \lb{m}$.
\end{definition}
We will denote by $\Bider(\lb{m})$ the set of all biderivations of $\lb{m}$. It is a Leibniz algebra with the obvious $\k$-module structure and the Leibniz bracket given by
\begin{equation*}\label{bracket_bider_m}
[(d_1,D_1), (d_2,D_2)] = (d_1 d_2 - d_2 d_1, D_1 d_2 - d_2 D_1).
\end{equation*}

It is not difficult to check that, given an element $m\in \lb{m}$, the pair $(\ad(m), \Ad(m))$, with $\ad(m)(m')=-[m',m]$ and $\Ad(m)(m')=[m,m']$ for all $m' \in \lb{m}$, is a biderivation. The pair $(\ad(m), \Ad(m))$ is called \emph{inner biderivation} of $m$.

\section{The main construction}\label{section3}
In this section we extend to crossed modules the Leibniz algebra of biderivations. First we need to translate the notion of a biderivation of a Leibniz algebra into a biderivation between two Leibniz algebras via the action.
\begin{definition}\label{def_bider_qn}
	Given an action of Leibniz algebras of $\lb{q}$ on $\lb{n}$, the set of \emph{biderivations} from $\lb{q}$ to $\lb{n}$, denoted by $\Bider(\lb{q},\lb{n})$, consists of all the pairs $(d,D)$ of $\k$-linear maps, $d,D \colon \lb{q} \to \lb{n}$, such that
	\begin{align}
	d([q,q']) & = [d(q),q'] + [q,d(q')], \label{axiom_bider_qn_1}\\
	D([q,q']) & = [D(q),q'] - [D(q'),q], \label{axiom_bider_qn_2}\\
	[q,d(q')] & = [q,D(q')], \label{axiom_bider_qn_3}
	\end{align}
	\noindent for all $q,q' \in \lb{q}$.
\end{definition}

Given $n \in \lb{n}$, the pair of $\k$-linear maps $(\ad(n),\Ad(n))$, where	$\ad(n) (q) = - [q,n]$ and	$\Ad(n) (q) = [n,q]$ for all $q \in \lb{q}$, is clearly a biderivation from $\lb{q}$ to $\lb{n}$. Observe that $\Bider(\lb{q},\lb{q})$, with the action of $\lb{q}$ on itself defined by its Leibniz bracket, is exactly $\Bider(\lb{q})$.

Let us assume for the rest of the article that $(\lb{n},\lb{q},\mu)$ is a Leibniz crossed module. One can easily check the following result.
\begin{lemma}\label{lemma_bider_from_bider_qn}
	Let $(d,D) \in \Bider(\lb{q},\lb{n})$. Then $(d \mu, D \mu) \in \Bider(\lb{n})$ and $(\mu d, \mu D) \in \Bider(\lb{q})$.
\end{lemma}
We also have the following result.
\begin{lemma}\label{lemma_bider_qn}
	Let $(d_1,D_1)$, $(d_2,D_2) \in \Bider(\lb{q},\lb{n})$. Then
	\begin{align*}
	[D_1 \mu d_2 (q),  q'] & = [D_1 \mu D_2(q), q'], \\
	[q, D_1 \mu d_2 (q')] & = [q, D_1 \mu D_2 (q')],
	\end{align*}
	\noindent for all $q,q' \in \lb{q}$.
\end{lemma}
\begin{proof}
	Let $q, q' \in \lb{q}$ and $(d_1,D_1), (d_2,D_2) \in \Bider(\lb{q},\lb{n})$. According to the identity \eqref{axiom_bider_qn_3} for $(d_2,D_2)$, $[q',d_2(q)]=[q',D_2(q)]$, so $D_1 \mu([q',d_2(q)])=D_1 \mu([q',D_2(q)])$. Due to \eqref{axiom_bider_qn_2} and the equivariance of $(\lb{q}, \lb{n}, \mu)$, one can easily derive that
	\begin{equation*}
	[D_1(q'),\mu d_2(q)] - [D_1 \mu d_2(q),q'] = [D_1(q'),\mu D_2(q)] - [D_1 \mu D_2(q),q'].
	\end{equation*}
	\noindent By the Peiffer identity and \eqref{axiom_bider_qn_3} for $(d_2,D_2)$, $[D_1(q'),\mu d_2(q)] = [D_1(q'),\mu D_2(q)]$. Therefore $[D_1 \mu d_2(q),q'] = [D_1 \mu D_2(q),q']$.
	
	The other identity can be proved similarly by using \eqref{axiom_bider_qn_1} and \eqref{axiom_bider_qn_3}.
\end{proof}

$\Bider(\lb{q},\lb{n})$ has an obvious $\k$-module structure. Regarding its Leibniz structure, it is described in the next proposition.
\begin{proposition}
	$\Bider(\lb{q},\lb{n})$ is a Leibniz algebra with the bracket given by
	\begin{equation}\label{eq_bider_qn_bracket}
	[(d_1,D_1),(d_2,D_2)] = (d_1 \mu d_2 - d_2 \mu d_1, D_1 \mu d_2 - d_2 \mu D_1)
	\end{equation}
	\noindent for all $(d_1,D_1), (d_2,D_2) \in \Bider(\lb{q},\lb{n})$.
\end{proposition}
\begin{proof}
It follows directly from Lemma~\ref{lemma_bider_qn}.
\end{proof}
Now we state the following definition.
\begin{definition}\label{def_bider_nqmu}
	The set of \emph{biderivations of the Leibniz crossed module} $(\lb{n},\lb{q},\mu)$, denoted by $\Bider(\lb{n},\lb{q},\mu)$, consists of all quadruples $((\sigma_1,\theta_1),(\sigma_2,\theta_2))$ such that
	\begin{align}
	& (\sigma_1,\theta_1) \in \Bider(\lb{n}) \quad \text{and} \quad (\sigma_2,\theta_2) \in \Bider(\lb{q}),\label{axiom_bider_nqmu_1}\\
	& \mu \sigma_1 = \sigma_2 \mu \quad \text{and} \quad \mu \theta_1 = \theta_2 \mu,\label{axiom_bider_nqmu_2}\\
	& \sigma_1([q,n]) = [\sigma_2(q),n] + [q,\sigma_1(n)],\label{axiom_bider_nqmu_3}\\
	& \sigma_1([n,q]) = [\sigma_1(n),q] + [n,\sigma_2(q)],\label{axiom_bider_nqmu_4}\\
	& \theta_1([q,n]) = [\theta_2(q),n] - [\theta_1(n),q],\label{axiom_bider_nqmu_5}\\
	& \theta_1([n,q]) = [\theta_1(n),q] - [\theta_2(q),n],\label{axiom_bider_nqmu_6}\\
	& [q,\sigma_1(n)] = [q,\theta_1(n)],\label{axiom_bider_nqmu_7}\\
	& [n,\sigma_2(q)] = [n,\theta_2(q)],\label{axiom_bider_nqmu_8}
	\end{align}
	\noindent for all $n \in \lb{n}$, $q \in \lb{q}$.
\end{definition}
Given $q \in \lb{q}$, it can be readily checked that $((\sigma^{q}_1,\theta^{q}_1),(\sigma^{q}_2,\theta^{q}_2))$, where
\begin{align*}
\sigma^{q}_1 (n) & = - [n,q], \qquad \theta^{q}_1 (n)  = [q,n],\\
\sigma^{q}_2 (q') & = - [q',q], \qquad \theta^{q}_2 (q')  = [q,q'],
\end{align*}
is a biderivation of the crossed module $(\lb{n}, \lb{q}, \mu)$.

The following lemma is necessary in order to prove that $\Bider(\lb{n},\lb{q},\mu)$ is indeed a Leibniz algebra.
\begin{lemma}\label{lemma_bider_nqmu}
	Let $((\sigma_1,\theta_1),(\sigma_2,\theta_2)), ((\sigma'_1,\theta'_1),(\sigma'_2,\theta'_2)) \in \Bider(\lb{n},\lb{q},\mu)$ and $(d,D) \in \Bider(\lb{q},\lb{n})$. Then
	\begin{equation*}
	\begin{split}
	[D \sigma_2 (q),  q'] & = [D \theta_2(q),  q'], \\
	[q, D \sigma_2 (q')] & = [q, D \theta_2 (q')], \\
	[\theta_1 d (q),  q'] & = [\theta_1 D (q),  q'], \\
	[q, \theta_1 d (q')] & = [q, \theta_1 D (q')],
	\end{split}
	\qquad
	\begin{split}
	[D \sigma_2 (q),  n] & = [D \theta_2(q),  n], \\
	[n, D \sigma_2 (q)] & = [n, D \theta_2 (q)], \\
	[\theta_1 d (q),  n] & = [\theta_1 D (q),  n], \\
	[n, \theta_1 d (q)] & = [n, \theta_1 D (q)],
	\end{split}
	\qquad
	\begin{split}
	[\theta_1 \sigma'_1 (n),  q] & = [\theta_1 \theta'_1(n),  q], \\
	[q, \theta_1 \sigma'_1 (n)] & = [q, \theta_1 \theta'_1(n)], \\
	[\theta_2 \sigma'_2 (q),  n] & = [\theta_2 \theta'_2(q),  n], \\
	[n, \theta_2 \sigma'_2 (q)] & = [n, \theta_2 \theta'_2(q)],
	\end{split}
	\end{equation*}
	\noindent for all $n \in \lb{n}$, $q,q' \in \lb{q}$.
\end{lemma}
\begin{proof}
	Let us show how to prove the first identity; the rest of them can be checked similarly. Let $q, q' \in \lb{q}$, $(d,D) \in \Bider(\lb{q},\lb{n})$ and $((\sigma_1,\theta_1),(\sigma_2,\theta_2)) \in \Bider(\lb{n},\lb{q},\mu)$. Since $(\sigma_2, \theta_2)$ is a biderivation of $\lb{q}$, we have that $[q',\sigma_2(q)]=[q',\theta_2(q)]$. Therefore $D([q',\sigma_2(q)])=D([q',\theta_2(q)])$. Directly from \eqref{axiom_bider_qn_2}, we get that
	\begin{equation*}
	[D(q'),\sigma_2(q)] - [D \sigma_2(q),q'] = [D(q'),\theta_2(q)] - [D \theta_2(q),q'].
	\end{equation*}
	\noindent Thus, due to \eqref{axiom_bider_nqmu_8}, $[D(q'),\sigma_2(q)] = [D(q'),\theta_2(q)]$. Hence, $[D \sigma_2(q),q'] = [D \theta_2(q),q']$.
\end{proof}
The $\k$-module structure of $\Bider(\lb{n}, \lb{q}, \mu)$ is evident, while its Leibniz structure is described as follows.
\begin{proposition}
	$\Bider(\lb{n}, \lb{q}, \mu)$ is a Leibniz algebra with the bracket given by
	\begin{multline}\label{eq_bider_nqmu_bracket}
	[((\sigma_1,\theta_1),(\sigma_2,\theta_2)), ((\sigma'_1,\theta'_1),(\sigma'_2,\theta'_2))] = ([(\sigma_1,\theta_1),(\sigma'_1,\theta'_1)],[(\sigma_2,\theta_2),(\sigma'_2,\theta'_2)]) \\
	= ((\sigma_1 \sigma'_1 - \sigma'_1 \sigma_1, \theta_1 \sigma'_1 - \sigma'_1 \theta_1),(\sigma_2 \sigma'_2 - \sigma'_2 \sigma_2, \theta_2 \sigma'_2 - \sigma'_2 \theta_2)),
	\end{multline}
	\noindent for all $((\sigma_1,\theta_1),(\sigma_2,\theta_2)), ((\sigma'_1,\theta'_1),(\sigma'_2,\theta'_2)) \in \Bider(\lb{n},\lb{q},\mu)$.
\end{proposition}
\begin{proof}
It follows directly from Lemma~\ref{lemma_bider_nqmu}.
\end{proof}
\begin{proposition}\label{prop_delta_lb_morph}
	The $\k$-linear map $\Delta \colon \Bider(\lb{q},\lb{n}) \to \Bider(\lb{n},\lb{q},\mu)$, given by $(d,D) \mapsto ((d \mu, D \mu),(\mu d, \mu D))$ is a homomorphism of Leibniz algebras.
\end{proposition}
\begin{proof}
$\Delta$ is well defined due to Lemma~\ref{lemma_bider_from_bider_qn}, while checking that it is a homomorphism of Leibniz algebras is a matter of straightforward calculations.
\end{proof}
	Since we aspire to make $\Delta$ into a Leibniz crossed module, we need to define an action of $\Bider(\lb{n},\lb{q},\mu)$ on $\Bider(\lb{q},\lb{n})$.
	\begin{theorem}\label{theo_action_bider_nqmu}
		There is an action of $\Bider(\lb{n},\lb{q},\mu)$ on $\Bider(\lb{q},\lb{n})$ given by:
		\begin{align}
		[((\sigma_1,\theta_1),(\sigma_2,\theta_2)),(d,D)] & = (\sigma_1 d - d \sigma_2, \theta_1 d - d \theta_2), \label{eq_action_Xbider_left}\\
		[(d,D),((\sigma_1,\theta_1),(\sigma_2,\theta_2))] & = (d \sigma_2 - \sigma_1 d, D \sigma_2 - \sigma_1 D),\label{eq_action_Xbider_right}
		\end{align}
		\noindent for all $((\sigma_1,\theta_1),(\sigma_2,\theta_2)) \in \Bider(\lb{n},\lb{q},\mu)$, $(d,D) \in \Bider(\lb{q},\lb{n})$. Moreover, the Leibniz homomorphism $\Delta$ (see Proposition~\ref{prop_delta_lb_morph}) together with the above action is a Leibniz crossed module.
	\end{theorem}
	\begin{proof}
		Let $(d,D) \in \Bider(\lb{q},\lb{n})$ and $((\sigma_1,\theta_1),(\sigma_2,\theta_2)) \in \Bider(\lb{n},\lb{q},\mu)$. Checking that both $(\sigma_1 d - d \sigma_2, \theta_1 d - d \theta_2)$ and $(d \sigma_2 - \sigma_1 d, D \sigma_2 - \sigma_1 D)$ satisfy conditions \eqref{axiom_bider_qn_1} and \eqref{axiom_bider_qn_2} requires the combined use of the properties satisfied by the elements in $\Bider(\lb{n},\lb{q},\mu)$ and $(d,D)$, but calculations are fairly straightforward. As an example, we show how to prove that $(\sigma_1 d - d \sigma_2, \theta_1 d - d \theta_2)$ verifies \eqref{axiom_bider_qn_1}. Let $q, q' \in \lb{q}$. Then
		\begin{align*}
		(\sigma_1 d - d \sigma_2) ([q,q'])  = & \sigma_1 ([d(q),q'] + [q,d(q')]) - d([\sigma_2(q), q'] + [q,\sigma_2(q')]) \\
		 = & [\sigma_1 d (q), q'] + [d(q), \sigma_2(q')] + [\sigma_2(q),d(q')] + [q,\sigma_1 d(q')] \\
		& - [d \sigma_2(q), q'] - [\sigma_2(q), d(q')] - [d(q),\sigma_2(q')] - [q,d \sigma_2(q')] \\
		 = & [(\sigma_1 d - d \sigma_2) (q), q'] + [q, (\sigma_1 d - d \sigma_2) (q')].
		\end{align*}
		As for condition \eqref{axiom_bider_qn_3}, in the case of $(\sigma_1 d - d \sigma_2, \theta_1 d - d \theta_2)$, it follows from \eqref{axiom_bider_nqmu_7}, the identity \eqref{axiom_bider_qn_3} for $(d,D)$ and the second identity in the first column from Lemma~\ref{lemma_bider_nqmu}. Namely,
		\begin{align*}
		[q, (\sigma_1 d - d \sigma_2) (q')] & = [q, \sigma_1 d(q')] - [q, d \sigma_2(q')] = [q, \theta_1 d(q')] - [q, D \sigma_2(q')] \\
		& = [q, \theta_1 d(q')] - [q, D \theta_2(q')] = [q, \theta_1 d(q')] - [q, d \theta_2(q')],
		\end{align*}
		\noindent for all $q,q' \in \lb{q}$. A similar procedure allows to prove that $(d \sigma_2 - \sigma_1 d, D \sigma_2 - \sigma_1 D)$ satisfies condition \eqref{axiom_bider_qn_3} as well.
		
		Routine calculations show that \eqref{eq_action_Xbider_left} and \eqref{eq_action_Xbider_right} together with the definition of the brackets in $\Bider(\lb{n},\lb{q},\mu)$ and $\Bider(\lb{q},\lb{n})$ provide an action of Leibniz algebras.
		
		It only remains to prove that $\Delta$ satisfies the equivariance and the Peiffer identity. It is immediate to check that
		\begin{multline}\label{equiv_Xlb_actor_1}
		\Delta ([((\sigma_1,\theta_1),(\sigma_2,\theta_2)),(d,D)]) = ((\sigma_1 d \mu - d \sigma_2 \mu, \theta_1 d \mu - d \theta_2 \mu),\\ (\mu \sigma_1 d - \mu d \sigma_2, \mu \theta_1 d - \mu d \theta_2)),
		\end{multline}
		\noindent while
		\begin{multline}\label{equiv_Xlb_actor_2}
		[((\sigma_1,\theta_1),(\sigma_2,\theta_2)),\Delta(d,D)] = ((\sigma_1 d \mu - d \mu \sigma_1, \theta_1 d \mu - d \mu \theta_1),\\ (\sigma_2 \mu d - \mu d \sigma_2, \theta_2 \mu d - \mu d \theta_2)).
		\end{multline}
		Condition \eqref{axiom_bider_nqmu_2} guarantees that $\eqref{equiv_Xlb_actor_1} = \eqref{equiv_Xlb_actor_2}$. The other identity can be checked similarly. The Peiffer identity follows immediately from \eqref{eq_action_Xbider_left} and \eqref{eq_action_Xbider_right} along the definition of $\Delta$ and the bracket in $\Bider(\lb{q},\lb{n})$.
	\end{proof}

\section{The actor}\label{section4}

In \cite{Or}, Orzech introduced the notion of category of interest, which is nothing but a category of groups with operations verifying two extra conditions. $\Lb$ is a category of interest, although $\XLb$ is not. Nevertheless, it is equivalent to the category of $cat^1$-Leibniz algebras (see for example \cite{CaKhLa}), which is itself a modified category of interest in the sense of \cite{BoCaDaUs}.  So it makes sense to study representability of actions in $\XLb$ under the context of modified categories of interest, as it is done in \cite{BoCaDaUs} for crossed modules of associative algebras.
However, since $\XLb$ is an example of semi-abelian categories, and an action is the same as a split extension  in any semi-abelian category \cite[Lemma 1.3]{BoJaKe}, we choose a different, more combinatorial approach to the problem, by constructing the semidirect product (split extension) of Leibniz crossed modules.

We use the term \emph{actor} (as in \cite{BoCaDaUs, CaDaLa}) for an object which represents actions in a semi-abelian category, the general definition of which is known from \cite{BoJaKe} under the name \emph{split extension classifier}.

\

We need to remark that, given a Leibniz algebra $\lb{m}$, $\Bider(\lb{m})$ is the actor of $\lb{m}$ under certain conditions. In particular, the following result is proved in  \cite{CaDaLa}.
\begin{proposition}[\cite{CaDaLa}]\label{prop_Lb_ann_perfect}
	Let $\lb{m}$ be a Leibniz algebra such that $\Ann(\lb{m})=0$ or $[\lb{m},\lb{m}]=\lb{m}$. Then $\Bider(\lb{m})$ is the actor of $\lb{m}$.
\end{proposition}

	Bearing in mind the ease of the generalization of the actor in the category of groups and Lie algebras to crossed modules, together with the role of $\Bider(\lb{m})$ in regard to any Leibniz algebra $\lb{m}$, it makes sense to consider $(\Bider(\lb{q},\lb{n}),\Bider(\lb{n},\lb{q},\mu),\Delta)$ as a candidate for actor in $\XLb$, at least under certain conditions (see Proposition~\ref{prop_Lb_ann_perfect}). However, it would be reckless to define an action of a Leibniz crossed module $(\lb{m},\lb{p},\eta)$ on $(\lb{n},\lb{q},\mu)$ as a homomorphism from $(\lb{m},\lb{p},\eta)$ to the Leibniz crossed module $(\Bider(\lb{q},\lb{n}),\Bider(\lb{n},\lb{q},\mu),\Delta)$, since we cannot ensure that the mentioned homomorphism induces a set of actions of $(\lb{m},\lb{p},\eta)$ on $(\lb{n},\lb{q},\mu)$ from which we can construct the semidirect product.
	
	In \cite[Proposition 2.1]{CaInKhLa} the authors give an equivalent description of an action of a crossed module of groups in terms of equations. A similar description can be done for an action of a crossed module of Lie algebras (see \cite{CaCaKhLa}). This determines our approach to the problem. We consider a homomorphism from a Leibniz crossed module $(\lb{m},\lb{p},\eta)$ to $(\Bider(\lb{q},\lb{n}),\Bider(\lb{n},\lb{q},\mu),\Delta)$, which will be denoted by $\ol{\Act}(\lb{n},\lb{q},\mu)$ from now on, and unravel all the properties satisfied by the mentioned homomorphism, transforming them into a set of equations. Then we check that the existence of that set of equations is equivalent to the existence of a homomorphism of Leibniz crossed modules from $(\lb{m},\lb{p},\eta)$ to $\ol{\Act}(\lb{n},\lb{q},\mu)$ only under certain conditions. Finally we prove that those equations indeed describe a comprehensive set of actions by constructing the associated semidirect product, which is an object in $\XLb$.	
		\begin{lemma}\label{lemma_bider_combined}\hfill
			\begin{enumerate}
				\item [\rm (i)] Let $\lb{q}$ be a Leibniz algebra and $(\sigma,\theta), (\sigma',\theta') \in \Bider(\lb{q})$. If $\Ann(\lb{q}) = 0$ or $[\lb{q},\lb{q}] = \lb{q}$, then
				\begin{equation}\label{eq_ann_0_Lb}
				\theta \sigma' (q) = \theta \theta' (q),
				\end{equation}
				\noindent for all $q \in \lb{q}$.
				\item [\rm (ii)] Let $(\lb{n},\lb{q},\mu)$ be a Leibniz crossed module, $((\sigma_1,\theta_1),(\sigma_2,\theta_2)) \in \Bider(\lb{n},\lb{q},\mu)$ and $(d,D) \in \Bider(\lb{q},\lb{n})$. If $\Ann(\lb{n}) = 0$ or $[\lb{q},\lb{q}] = \lb{q}$, then
				\begin{align}
				D \sigma_2 (q) & = D \theta_2 (q), \label{eq_ann_0_XLb_1} \\
				\theta_1 d (q) & = \theta_1 D (q), \label{eq_ann_0_XLb_2}
				\end{align}
				\noindent for all $q \in \lb{q}$.
			\end{enumerate}
		\end{lemma}
		\begin{proof}
			Calculations in order to prove (i) are straightforward. Regarding (ii), $D \sigma_2 (q) - D \theta_2 (q)$ and $\theta_1 d (q) - \theta_1 D (q)$ are elements in $\Ann(\lb{n})$, immediately from the identities in the second column from Lemma~\ref{lemma_bider_nqmu}. Therefore, if $\Ann(\lb{n}) = 0$, it is clear that \eqref{eq_ann_0_XLb_1} and \eqref{eq_ann_0_XLb_2} hold.
			
			Let us now assume that $[\lb{q},\lb{q}] = \lb{q}$. Given $q,q' \in \lb{q}$, directly from the fact that $(\sigma_2,\theta_2) \in \Bider(\lb{q})$ and $(d,D) \in \Bider(\lb{q},\lb{n})$, we get that
			\begin{align*}
			D \theta_2([q,q']) & = [D \theta_2(q),q'] - [D(q'),\theta_2(q)] - [D \theta_2 (q'),q] + [D (q),\theta_2 (q')], \\
			D \sigma_2([q,q']) & = [D \sigma_2(q),q'] - [D(q'),\sigma_2(q)] + [D(q),\sigma_2(q')] - [D \sigma_2(q'),q].
			\end{align*}
			 Due to \eqref{axiom_bider_nqmu_8} and the first identity in the first column from Lemma~\ref{lemma_bider_nqmu}, $D \theta_2([q,q']) = D \sigma_2([q,q'])$. By hypothesis, every element in $\lb{q}$ can be expressed as a linear combination of elements of the form $[q,q']$. This fact together with the linearity of $D$, $\sigma_2$ and $\theta_2$, guarantees that $D \theta_2(q) = D \sigma_2(q)$ for all $q \in \lb{q}$. The identity \eqref{eq_ann_0_XLb_2} can be checked similarly by making use of \eqref{axiom_bider_qn_3}, \eqref{axiom_bider_nqmu_5}, \eqref{axiom_bider_nqmu_6} and the third identity in the first column from Lemma~\ref{lemma_bider_nqmu}.
		\end{proof}

	\begin{theorem}\label{theo_equiv_XLb_action}
		Let $(\lb{m},\lb{p},\eta)$ and $(\lb{n},\lb{q},\mu)$ in $\XLb$. There exists a homomorphism of crossed modules from $(\lb{m},\lb{p},\eta)$ to $(\Bider(\lb{q},\lb{n}), \Bider(\lb{n},\lb{q},\mu), \Delta)$, if the following conditions hold:
		\begin{itemize}
			\item [\rm (i)] There are actions of the Leibniz algebra $\lb{p}$ (and so $\lb{m}$) on the Leibniz algebras $\lb{n}$ and $\lb{q}$. The homomorphism $\mu$
			is $\lb{p}$-equivariant, that is
			\begin{align}
			\mu([p,n]) & =[p,\mu(n)], \label{p_equivariant_Lb_1} \tag{LbEQ1} \\
			\mu([n,p]) & =[\mu(n),p], \label{p_equivariant_Lb_2} \tag{LbEQ2}
			\end{align}
			\noindent and the actions of $\lb{p}$ and $\lb{q}$ on $\lb{n}$ are compatible, that is
			\begin{align}
			[n,[p,q]] & = [[n,p],q]-[[n,q],p], \label{comp_of_actions_Lb_1} \tag{LbCOM1} \\
			[p,[n,q]] & = [[p,n],q]-[[p,q],n], \label{comp_of_actions_Lb_2} \tag{LbCOM2} \\
			[p,[q,n]] & = [[p,q],n]-[[p,n],q], \label{comp_of_actions_Lb_3} \tag{LbCOM3} \\
			[n,[q,p]] & = [[n,q],p]-[[n,p],q], \label{comp_of_actions_Lb_4} \tag{LbCOM4} \\
			[q,[n,p]] & = [[q,n],p]-[[q,p],n], \label{comp_of_actions_Lb_5} \tag{LbCOM5} \\
			[q,[p,n]] & = [[q,p],n]-[[q,n],p], \label{comp_of_actions_Lb_6} \tag{LbCOM6}
			\end{align}
			for all $n\in \lb{n}$, $p\in \lb{p}$ and $q\in \lb{q}$.		
			\item[\rm (i)] There are two $\k$-bilinear maps $\xi_1 \colon \lb{m}\times \lb{q}\to \lb{n}$ and $\xi_2 \colon \lb{q}\times \lb{m}\to \lb{n}$ such that
			\begin{align}
			\mu \xi_2(q,m) & =[q,m],\label{action_Lb_1a} \tag{LbM1a} \\
			\mu \xi_1(m,q) & =[m,q],\label{action_Lb_1b} \tag{LbM1b} \\
			\xi_2 (\mu(n),m) & = [n,m],\label{action_Lb_2a} \tag{LbM2a} \\
			\xi_1 (m,\mu(n)) & = [m,n],\label{action_Lb_2b} \tag{LbM2b} \\
			\xi_2 (q,[p,m]) & = \xi_2 ([q,p],m) - [\xi_2(q,m),p], \label{action_Lb_3a} \tag{LbM3a} \\
			\xi_1 ([p,m],q) & = \xi_2 ([p,q],m) - [p,\xi_2(q,m)], \label{action_Lb_3b} \tag{LbM3b} \\
			\xi_2 (q,[m,p]) & = [\xi_2(q,m),p] - \xi_2 ([q,p],m), \label{action_Lb_3c} \tag{LbM3c} \\
			\xi_1 ([m,p],q) & = [\xi_1(m,q),p] - \xi_1 (m,[q,p]), \label{action_Lb_3d} \tag{LbM3d} \\
			\xi_2 (q,[m,m']) & = [\xi_2(q,m),m'] - [\xi_2(q,m'),m], \label{action_Lb_4a} \tag{LbM4a} \\
			\xi_1 ([m,m'],q) & = [\xi_1(m,q),m'] - [m,\xi_2(q,m')], \label{action_Lb_4b} \tag{LbM4b} \\
			\xi_2 ([q,q'],m) & = [\xi_2(q,m),q'] + [q,\xi_2(q',m)], \label{action_Lb_5a} \tag{LbM5a} \\
			\xi_1 (m,[q,q']) & = [\xi_1(m,q),q'] - [\xi_1(m,q'),q], \label{action_Lb_5b} \tag{LbM5b} \\
			[q,\xi_1(m,q')] & = - [q,\xi_2(q',m)], \label{action_Lb_5c} \tag{LbM5c} \\
			\xi_1 (m,[p,q]) &= - \xi_1(m,[q,p]), \label{action_Lb_6a} \tag{LbM6a} \\
			[p,\xi_1(m,q)] & = - [p,\xi_2(q,m)], \label{action_Lb_6b} \tag{LbM6b}
			\end{align}
			\noindent for all $m,m'\in \lb{m}$, $n\in \lb{n}$, $p\in \lb{p}$, $q,q'\in \lb{q}$.
		\end{itemize}
		
		Additionally, the converse statement is also true if one of the following conditions holds:
		\begin{align}
		& \Ann(\lb{n})=0=\Ann(\lb{q}), \label{condition1} \tag{CON1} \\
		& \Ann(\lb{n})=0 \quad \text{and} \quad [\lb{q},\lb{q}] = \lb{q}, \label{condition2} \tag{CON2} \\
		& [\lb{n},\lb{n}] = \lb{n} \quad \text{and} \quad [\lb{q},\lb{q}] = \lb{q}. \label{condition3} \tag{CON3}
		\end{align}
	\end{theorem}
		\begin{proof}
			Let us suppose that (i) and (ii) hold. It is possible to define a homomorphism of crossed modules $(\vp,\psi)$ from $(\lb{m},\lb{p},\eta)$ to $\ol{\Act}(\lb{n},\lb{q},\mu)$ as follows. Given $m \in \lb{m}$, $\vp(m)=(d_{m},D_{m})$, with
			\begin{equation*}
			d_{m} (q) = - \xi_{2}(q,m), \qquad D_{m} (q) = \xi_{1}(m,q),
			\end{equation*}
			\noindent for all $q \in \lb{q}$. On the other hand, for any $p \in \lb{p}$, $\psi(p)=((\sigma^{p}_1,\theta^{p}_1), (\sigma^{p}_2,\theta^{p}_2))$, with
			\begin{align*}
			\sigma^{p}_1(n) & = - [n,p],  \qquad \theta^{p}_1(n) = [p,n], \\
			\sigma^{p}_2(q) & = - [q,p],  \qquad \theta^{p}_2(q) = [p,q],
			\end{align*}
			\noindent for all $n \in \lb{n}$, $q \in \lb{q}$.
			It follows directly from \eqref{action_Lb_5a}--\eqref{action_Lb_5c} that $(d_{m},D_{m}) \in \Bider(\lb{q},\lb{n})$ for all $m \in \lb{m}$. Besides, $\vp$ is clearly $\k$-linear and given $m,m' \in \lb{m}$,
			\begin{equation*}
			[\vp(m),\vp(m')] = [(d_{m},D_{m}),(d_{m'},D_{m'})] = [d_{m} \mu d_{m'} - d_{m'} \mu d_{m}, D_{m} \mu d_{m'} - d_{m'} \mu D_{m}].
			\end{equation*}
			\noindent For any $q \in \lb{q}$,
			\begin{align*}
			d_{m} \mu d_{m'} (q) - d_{m'} \mu d_{m} (q) & = -\xi_{2} (\mu d_{m'} (q), m) + \xi_2 (\mu d_{m} (q), m') \\
			& = - [d_{m'} (q),m] + [d_{m}(q),m'] \\ & = [\xi_{2} (q,m'),m] - [\xi_{2}(q,m), m'] \\
			& = -\xi_{2} (q,[m,m']) = d_{[m,m']}(q),
			\end{align*}
			\noindent due to \eqref{action_Lb_2a} and \eqref{action_Lb_4a}. Analogously, it can be easily checked the identity $(D_{m} \mu d_{m'} - d_{m'} \mu D_{m}) (q) = D_{[m,m']} (q)$ by making use of \eqref{action_Lb_2a}, \eqref{action_Lb_2b} and \eqref{action_Lb_4b}. Hence, $\vp$ is a homomorphism of Leibniz algebras.
			
			As for $\psi$, it is necessary to prove that $((\sigma^{p}_1,\theta^{p}_1), (\sigma^{p}_2,\theta^{p}_2))$ satisfies all the axioms from Definition~\ref{def_bider_nqmu} for any $p \in \lb{p}$.  The fact that $(\sigma^{p}_1,\theta^{p}_1)$ (respectively $(\sigma^{p}_2,\theta^{p}_2)$) is a biderivation of $\lb{n}$ (respectively $\lb{q}$) follows directly from the actions of $\lb{p}$ on $\lb{n}$ and $\lb{q}$. The identities $\mu \theta^{p}_1 = \theta^{p}_2 \mu$ and $\mu \sigma^{p}_1 = \sigma^{p}_2 \mu$ are immediate consequences of \eqref{p_equivariant_Lb_1} and \eqref{p_equivariant_Lb_2} respectively.
			
			Observe that the combinations of the identities \eqref{comp_of_actions_Lb_1} and \eqref{comp_of_actions_Lb_4} and the identities \eqref{comp_of_actions_Lb_5} and \eqref{comp_of_actions_Lb_6} yield the equalities
			\begin{equation*}
			-[n,[q,p]] = [n,[p,q]] \qquad \text{and} \qquad -[q,[n,p]] = [q,[p,n]].
			\end{equation*}
			\noindent These together with \eqref{comp_of_actions_Lb_2}--\eqref{comp_of_actions_Lb_5} allow us to prove that $((\sigma^{p}_1,\theta^{p}_1), (\sigma^{p}_2,\theta^{p}_2))$ does satisfy conditions \eqref{axiom_bider_nqmu_3}--\eqref{axiom_bider_nqmu_8} from Definition~\ref{def_bider_nqmu}. Therefore, $\psi$ is well defined, while it is obviously $\k$-linear. Moreover, due to \eqref{eq_bider_nqmu_bracket} we know that
			\begin{equation*}
			[\psi(p),\psi(p')] = ((\sigma^{p}_1 \sigma^{p'}_1 - \sigma^{p'}_1 \sigma^{p}_1, \theta^{p}_1 \sigma^{p'}_1 - \sigma^{p'}_1 \theta^{p}_1),(\sigma^{p}_2 \sigma^{p'}_2 - \sigma^{p'}_2 \sigma^{p}_2, \theta^{p}_2 \sigma^{p'}_2 - \sigma^{p'}_2 \theta^{p}_2)),
			\end{equation*}
			\noindent and by definition
			\begin{equation*}
			\psi([p,p']) = ((\sigma^{[p,p']}_1,\theta^{[p,p']}_1), (\sigma^{[p,p']}_2,\theta^{[p,p']}_2)).
			\end{equation*}
			\noindent One can easily check that the corresponding components are equal by making use of the actions of $\lb{p}$ on $\lb{n}$ and $\lb{q}$. Hence, $\psi$ is a homomorphism of Leibniz algebras.
			
			Recall that
			\begin{align*}
			\Delta \vp (m) & = ((d_{m} \mu, D_{m} \mu),(\mu d_{m}, \mu D_{m})), \\
			\psi\eta (m) & = ((\sigma^{\eta(m)}_1,\theta^{\eta(m)}_1),(\sigma^{\eta(m)}_2,\theta^{\eta(m)}_2)),
			\end{align*}
			\noindent for any $m \in \lb{m}$, but
			\begin{align*}
			d_m \mu (n) & = -\xi_2 (\mu(n),m) = -[n,m] = -[n,\eta(m)] = \sigma^{\eta(m)}_{1} (n),\\
			D_m \mu (n) & = \xi_1 (m,\mu(n)) = [m,n] = [\eta(m),n] = \theta^{\eta(m)}_{1} (n),\\
			\mu d_m (q) & = -\mu \xi_2(q,m) = - [q,m] = - [q,\eta(m)] = \sigma^{\eta(m)}_{2} (q), \\
			\mu D_m (q) & = \mu \xi_1(m,q) = [m,q] = [\eta(m),q] = \theta^{\eta(m)}_{2} (q),
			\end{align*}
			\noindent for all $n \in \lb{n}$, $q \in \lb{q}$, due to \eqref{action_Lb_1a}, \eqref{action_Lb_1b}, \eqref{action_Lb_2a}, \eqref{action_Lb_2b}. Therefore, $\Delta \vp = \psi \eta$.
			
			It only remains to check the behaviour of $(\vp,\psi)$ regarding the action of $\lb{p}$ on $\lb{m}$. Let $m \in \lb{m}$ and $p \in \lb{p}$. Due to \eqref{eq_action_Xbider_left} and \eqref{eq_action_Xbider_right},
			\begin{align*}
			[\psi(p),\vp(m)] & = (\sigma^{p}_1 d_{m} - d_{m} \sigma^{p}_2, \theta^{p}_1 d_{m} - d_{m} \theta^{p}_2), \\
			[\vp(m),\psi(p)] & = (d_{m} \sigma^{p}_2 - \sigma^{p}_1 d_{m}, D_{m} \sigma^{p}_2 - \sigma^{p}_1 D_{m}).
			\end{align*}
			On the other hand, by definition, we know that
			\begin{align*}
			\vp([p,m]) & = (d_{[p,m]},D_{[p,m]}), \\
			\vp([m,p]) & = (d_{[m,p]},D_{[m,p]}).
			\end{align*}
			\noindent Directly from \eqref{action_Lb_3a}, \eqref{action_Lb_3b}, \eqref{action_Lb_3c} and \eqref{action_Lb_3d} one can easily confirm that the required identities between components hold. Hence, we can finally ensure that $(\vp,\psi)$ is a homomorphism of Leibniz crossed modules.
			
			Now let us show that it is necessary that at least one of the conditions \eqref{condition1}--\eqref{condition3} holds in order to prove the converse statement. Let us suppose that there is a homomorphism of crossed modules
			\begin{equation}\label{action_diagramLb}
			\xymatrix {
				\lb{m} \ar[d]_{\vp}\ar[r]^{\eta} & \lb{p} \ar[d]^{\psi}  \\
				\Bider(\lb{q},\lb{n}) \ar[r]_{\Delta} & \Bider(\lb{n},\lb{q},\mu)
			}
			\end{equation}
			Given $m \in \lb{m}$ and $p \in \lb{p}$, let us denote $\vp(m)$ by $(d_{m},D_{m})$ and $\psi(p)$ by $((\sigma^{p}_1,\theta^{p}_1),(\sigma^{p}_2,\theta^{p}_2))$, which satisfy conditions \eqref{axiom_bider_qn_1}--\eqref{axiom_bider_qn_3} from Definition~\ref{def_bider_qn} and conditions \eqref{axiom_bider_nqmu_1}--\eqref{axiom_bider_nqmu_8} from Definition~\ref{def_bider_nqmu} respectively. Also, due to the definition of $\Delta$ (see Proposition~\ref{prop_delta_lb_morph}), the commutativity of \eqref{action_diagramLb} can be expressed by the identity
			\begin{equation}\label{commutativity_action_Lb}
			((d_{m} \mu,D_{m} \mu),(\mu d_{m},\mu D_{m})) = ((\sigma^{\eta(m)}_1,\theta^{\eta(m)}_1),(\sigma^{\eta(m)}_2,\theta^{\eta(m)}_2)),
			\end{equation}
			\noindent for all $m \in \lb{m}$.
			It is possible to define four bilinear maps, all of them denoted by $[-,-]$, from $\lb{p} \times \lb{n}$ to $\lb{n}$, $\lb{n} \times \lb{p}$ to $\lb{n}$, $\lb{p} \times \lb{q}$ to $\lb{q}$ and $\lb{q} \times \lb{p}$ to $\lb{q}$, given by
			\begin{align*}
			[p,n] & = \theta^{p}_1(n),  \qquad  [n,p] = -\sigma^{p}_1(n), \\
			[p,q] & = \theta^{p}_2(q),  \qquad  [q,p] = -\sigma^{p}_2(q),
			\end{align*}
			for all $n\in \lb{n}$, $p \in \lb{p}$, $q \in \lb{q}$. These maps define actions of $\lb{p}$ on $\lb{n}$ and $\lb{q}$. The first three identities for the action on $\lb{n}$ (respectively $\lb{q}$) follow easily from the fact that $(\sigma^{p}_1,\theta^{p}_1)$ (respectively $(\sigma^{p}_2,\theta^{p}_2)$) is a biderivation of $\lb{n}$ (respectively $\lb{q}$).
			
			Since $\psi$ is a Leibniz homomorphism, we get that
			\begin{multline*}
			((\sigma^{[p,p']}_1,\theta^{[p,p']}_1),(\sigma^{[p,p']}_2,\theta^{[p,p']}_2)) = ((\sigma^{p}_1 \sigma^{p'}_1 - \sigma^{p'}_1 \sigma^{p}_1, \theta^{p}_1 \sigma^{p'}_1 - \sigma^{p'}_1 \theta^{p}_1),\\(\sigma^{p}_2 \sigma^{p'}_2 - \sigma^{p'}_2 \sigma^{p}_2, \theta^{p}_2 \sigma^{p'}_2 - \sigma^{p'}_2 \theta^{p}_2)).
			\end{multline*}
			The identities between the first and the second (respectively the third and the fourth) components in those quadruples allow us to confirm the fourth and fifth identities for the action of $\lb{p}$ on $\lb{n}$ (respectively $\lb{q}$).
			
			As for the last condition for both actions, it is fairly straightforward to check that
			\begin{align*}
			[[p,p'],n] - [[p,n],p'] & = \theta^{p}_1 \sigma^{p'}_1 (n),\\
			[[p,p'],q] - [[p,q],p'] & = \theta^{p}_2 \sigma^{p'}_2 (q),
			\end{align*}
			\noindent while
			\begin{align*}
			[p,[p',n]] & = \theta^{p}_1 \theta^{p'}_1 (n),\\
			[p,[p',q]] & = \theta^{p}_2 \theta^{p'}_2 (q),
			\end{align*}
			\noindent for all $n \in \lb{n}$, $p,p' \in \lb{p}$, $q \in \lb{q}$. However, if at least one of the conditions \eqref{condition1}--\eqref{condition3} holds, due to Lemma~\ref{lemma_bider_combined}~(i), $\theta^{p}_1 \sigma^{p'}_1 (n)=\theta^{p}_1 \theta^{p'}_1 (n)$ and $\theta^{p}_2 \sigma^{p'}_2 (q)=\theta^{p}_2 \theta^{p'}_2 (q)$. Therefore, we can ensure that there are Leibniz actions of $\lb{p}$ on both $\lb{n}$ and $\lb{q}$, which induce actions of $\lb{m}$ on $\lb{n}$ and $\lb{q}$ via $\eta$.
			
			The reader might have noticed that a fourth possible condition on $(\lb{n},\lb{q},\mu)$ could have been considered in order to guarantee the existence of the actions of $\lb{p}$ on $\lb{n}$ and $\lb{q}$ from the existence of the homomorphism of Leibniz crossed modules $(\vp,\psi)$. In fact, if $[\lb{n},\lb{n}] = \lb{n}$ and $\Ann(\lb{q}) = 0$, the problem with the last condition for the actions could have been solved in the same way. Nevertheless, this fourth condition does not guarantee that (ii) holds, as we will prove immediately below.
			
			Regarding \eqref{p_equivariant_Lb_1} and \eqref{p_equivariant_Lb_2}, they follow directly from \eqref{axiom_bider_nqmu_2} (observe that, by hypothesis, $((\sigma^{p}_1,\theta^{p}_1),(\sigma^{p}_2,\theta^{p}_2))$ is a biderivation of $(\lb{n},\lb{q},\mu)$ for any $p \in \lb{p}$). Similarly, \eqref{comp_of_actions_Lb_1}--\eqref{comp_of_actions_Lb_6} follow almost immediately from \eqref{axiom_bider_nqmu_3}--\eqref{axiom_bider_nqmu_8}. Hence, (i) holds.
			
			Concerning (ii), we can define $\xi_1(m,q)=D_{m}(q)$ and $\xi_2(q,m)=-d_{m}(q)$ for any $m \in \lb{m}$, $q \in \lb{q}$. In this way, $\xi_1$ and $\xi_2$ are clearly bilinear. \eqref{action_Lb_1a}, \eqref{action_Lb_1b}, \eqref{action_Lb_2a} and \eqref{action_Lb_2b} follow immediately from the identity \eqref{commutativity_action_Lb} and the fact that the actions of $\lb{m}$ on $\lb{n}$ and $\lb{q}$ are induced by the actions of $\lb{p}$ via $\eta$.
			
			Identities \eqref{action_Lb_5a}, \eqref{action_Lb_5b} and \eqref{action_Lb_5c} are a direct consequence of \eqref{axiom_bider_qn_1}--\eqref{axiom_bider_qn_3} (recall that, by hypothesis, $(d_{m},D_{m})$ is a biderivation from $\lb{q}$ to $\lb{n}$ for any $m \in \lb{m}$).
			
			Since $\vp$ is a Leibniz homomorphism, we have that
			\begin{equation*}
			(d_{[m,m']}, D_{[m,m']}) = (d_{m} \mu d_{m'} - d_{m'} \mu d_{m}, D_{m} \mu d_{m'} - d_{m'} \mu D_{m}).
			\end{equation*}
			This identity, together with \eqref{action_Lb_2a} and \eqref{action_Lb_2b}, allows to easily prove that \eqref{action_Lb_4a} and \eqref{action_Lb_4b} hold.
			
			Note that, since $(\vp,\psi)$ is a homomorphism of Leibniz crossed modules, $\vp([p,m]) = [\psi(p),\vp(m)]$ and $\vp([m,p]) = [\vp(m),\psi(p)]$ for all $m \in \lb{m}$, $p \in \lb{p}$. Due to the definition of the action of $\Bider(\lb{n},\lb{q},\mu)$ on $\Bider(\lb{q},\lb{n})$ (see Theorem~\ref{theo_action_bider_nqmu}), we can write
			\begin{align*}
			(d_{[p,m]}, D_{[p,m]}) & = (\sigma^{p}_1 d_{m} - d_{m} \sigma^{p}_2, \theta^{p}_1 d_{m} - d_{m} \theta^{p}_2), \\
			(d_{[m,p]}, D_{[m,p]}) & = (d_{m} \sigma^{p}_2 - \sigma^{p}_1 d_{m}, D_{m} \sigma^{p}_2 - \sigma^{p}_1 D_{m}).
			\end{align*}
			Identities \eqref{action_Lb_3a}, \eqref{action_Lb_3b}, \eqref{action_Lb_3c} and \eqref{action_Lb_3d} follow immediately from the previous identities.
			
			Regarding \eqref{action_Lb_6a} and \eqref{action_Lb_6b}, directly from the definition of $\xi_1$, $\xi_2$ and the actions of $\lb{p}$ on $\lb{n}$ and $\lb{q}$, we have that
			\begin{align*}
			\xi_1(m,[p,q]) & = D_{m} \theta^{p}_2 (q), \qquad [p,\xi_1(m,q)] = \theta^{p}_1 D_{m} (q), \\
			- \xi_1(m,[q,p]) & = D_{m} \sigma^{p}_2 (q), \qquad - [p,\xi_2(q,m)] = \theta^{p}_1 d_{m} (q),
			\end{align*}
			\noindent for all $m \in \lb{m}$, $p\in \lb{p}$, $q \in \lb{q}$. Nevertheless, if at least one of the conditions \eqref{condition1}--\eqref{condition3} holds, due to Lemma~\ref{lemma_bider_combined}~(ii), $D_{m} \theta^{p}_2 (q) = D_{m} \sigma^{p}_2 (q)$ and $\theta^{p}_1 D_{m} (q) = \theta^{p}_1 d_{m} (q)$. Hence, (ii) holds.
		\end{proof}
		\begin{remark}
			A closer look at the proof of the previous theorem shows that neither conditions \eqref{action_Lb_6a} and \eqref{action_Lb_6b}, nor the identities $[p,[p',n]] = [[p,p'],n]-[[p,n],p']$ and $[p,[p',q]] = [[p,p'],q]-[[p,q],p']$ (which correspond to the sixth axiom satisfied by the actions of $\lb{p}$ on $\lb{n}$ and $\lb{q}$ respectively) are necessary in order to prove the existence of a homomorphism of crossed modules $(\vp,\psi)$ from $(\lb{m},\lb{p},\eta)$ to $\ol{\Act}(\lb{n},\lb{q},\mu)$, under the hypothesis that (i) and (ii) hold. Actually, if we remove those conditions from (i) and (ii), the converse statement would be true for any Leibniz crossed module $(\lb{n},\lb{q},\mu)$, even if it does not satisfy any of the conditions \eqref{condition1}--\eqref{condition3}. The problem is that \eqref{action_Lb_6a} and \eqref{action_Lb_6b}, together with the sixth identity satisfied by the actions of $\lb{p}$ on $\lb{n}$ and $\lb{q}$ are essential in order to prove that (i) and (ii) as in Theorem~\ref{theo_equiv_XLb_action} describe a set of actions of $(\lb{m},\lb{p},\eta)$ on $(\lb{n},\lb{q},\mu)$, as we will show immediately below. This agrees with the idea of $\ol{\Act}(\lb{n},\lb{q},\mu)$ not being ``good enough'' to be the actor of $(\lb{n},\lb{q},\mu)$ in general, just as $\Bider(\lb{m})$ is not always the actor of a Leibniz algebra $\lb{m}$.
		\end{remark}
		\begin{example}\label{canonical_morph}
			Let $(\lb{m},\lb{p},\eta) \in \XLb$, there is a homomorphism $(\vp,\psi) \colon (\lb{m},\lb{p},\eta) \to \ol{\Act}(\lb{m},\lb{p},\eta)$, with $\vp(m) = (d_{m},D_{m})$ and $\psi(p) = ((\sigma^{p}_1,\theta^{p}_1), (\sigma^{p}_2,\theta^{p}_2))$, where
			\begin{equation*}
			d_{m} (p) = - [p,m], \qquad D_{m} (p) = [m,p],
			\end{equation*}
			\noindent and
			\begin{align*}
			\sigma^{p}_1(m) & = - [m,p],  \qquad \theta^{p}_1(m) = [p,m], \\
			\sigma^{p}_2(p') & = - [p',p],  \qquad \theta^{p}_2(p') = [p,p'],
			\end{align*}
			\noindent for all $m \in \lb{m}$, $p,p' \in \lb{p}$. Calculations in order to prove that $(\vp,\psi)$ is indeed a homomorphism of Leibniz crossed modules are fairly straightforward. Of course, this homomorphism does not necessarily define a set of actions from which it is possible to construct the semidirect product. Theorem~\ref{theo_equiv_XLb_action}, along with the result immediately bellow, shows that if $(\lb{m},\lb{p},\eta)$ satisfies at least one of the conditions \eqref{condition1}--\eqref{condition3}, then the previous homomorphism does define an appropriate set of actions of $(\lb{m},\lb{p},\eta)$ on itself.
		\end{example}
		Let $(\lb{m},\lb{p},\eta)$ and $(\lb{n},\lb{q},\mu)$ be Leibniz crossed modules such that (i) and (ii) from Theorem~\ref{theo_equiv_XLb_action} hold. Therefore, there are Leibniz actions of $\lb{m}$ on $\lb{n}$ and of $\lb{p}$ on $\lb{q}$, so it makes sense to consider the semidirect products of Leibniz algebras $\lb{n} \rtimes \lb{m}$ and $\lb{q} \rtimes \lb{p}$. Furthermore, we have the following result.
		\begin{theorem}\label{theo_semi_XLb}
			There is an action of the Leibniz algebra $\lb{q}\rtimes \lb{p}$ on the Leibniz algebra $\lb{n}\rtimes \lb{m}$, given by
			\begin{align}
			[(q,p),(n,m)] = ([q,n] + [p,n] + \xi_2(q,m),[p,m]), \label{action_semidirect_XLb_left} \\
			[(n,m),(q,p)] = ([n,q] + [n,p] + \xi_1(m,q),[m,p]), \label{action_semidirect_XLb_right}
			\end{align}
			\noindent for all $(q,p) \in \lb{q} \rtimes \lb{p}$, $(n,m) \in \lb{n} \rtimes \lb{m}$, with $\xi_1$ and $\xi_2$ as in Theorem~\ref{theo_equiv_XLb_action}. Moreover, the Leibniz homomorphism $(\mu,\eta) \colon \lb{n}\rtimes \lb{m}\to \lb{q}\rtimes \lb{p}$, given by
			\[
			(\mu,\eta)(n,m)=(\mu(n),\eta(m)),
			\]
			\noindent for all $(n,m)\in \lb{n} \rtimes \lb{m}$, together with the previous action, is a Leibniz crossed module.
		\end{theorem}
			\begin{proof}
				Identities \eqref{action_semidirect_XLb_left} and \eqref{action_semidirect_XLb_right} follow easily from the conditions satisfied by $(\lb{m},\lb{p},\eta)$ and $(\lb{n},\lb{q},\mu)$ (see Theorem~\ref{theo_equiv_XLb_action}). Nevertheless, as an example, we show how to prove the third one. Calculations for the rest of the identities are similar. Let $(n,m),(n',m') \in \lb{n} \rtimes \lb{m}$ and $(q,p) \in \lb{q} \rtimes \lb{p}$. By routine calculations we get that
				\begin{multline*}
				[(n,m),[(n',m'),(q,p)]] = (\underbrace{[n,[n',q]]}_{(1)} \underbrace{+ [n,[n',p]]}_{(2)} \underbrace{+ [n,\xi_1(m',q)]}_{(3)} \underbrace{+ [m,[n',q]]}_{(4)} \\ \underbrace{+ [m,[n',p]]}_{(5)} \underbrace{+ [m,\xi_1(m',q)]}_{(6)} \underbrace{+ [n,[m',p]]}_{(7)}, \underbrace{[m,[m',p]]}_{(8)}),\\
				[[(n,m),(n',m')],(q,p)] = (\underbrace{[[n,n'],q]}_{(1')} \underbrace{+ [[n,n'],p]}_{(2')} \underbrace{+ [[n,m'],q]}_{(3')} \underbrace{+ [[m,n'],q]}_{(4')} \\ \underbrace{+ [[m,n'],p]}_{(5')} \underbrace{+ \xi_1([m,m'],q)}_{(6')} \underbrace{+ [[n,m'],p]}_{(7')}, \underbrace{[[m,m'],p]}_{(8')}),\\
				[[(n,m),(q,p)],(n',m')] = (\underbrace{[[n,q],n']}_{(1'')} \underbrace{+ [[n,p],n']}_{(2'')} \underbrace{+ [[n,q],m']}_{(3'')} \underbrace{+ [\xi_1(m,q),n']}_{(4'')} \\ \underbrace{+ [[m,p],n']}_{(5'')} \underbrace{+ [\xi_1(m,q),m']}_{(6'')} \underbrace{+ [[n,p],m']}_{(7'')}, \underbrace{[[m,p],m']}_{(8'')}).
				\end{multline*}
				Let us show that $(i) = (i') - (i'')$ for $i=1,\dots,8$. It is immediate for $i=1,2,8$ due to the action of $\lb{q}$ on $\lb{n}$ and the actions of $\lb{p}$ on $\lb{n}$ and $\lb{m}$. For $i=5$, the identity follows from the fact that the action of $\lb{m}$ on $\lb{n}$ is defined via $\eta$ together with the equivariance of $\eta$. Namely,
				\begin{align*}
				[m,[n',p]] & = [\eta(m),[n',p]] = [[\eta(m),n'],p] - [[\eta(m),p],n'] \\
				& = [[m,n'],p] - [\eta([m,p]),n'] = [[m,n'],p] - [[m,p],n'].
				\end{align*}
				The procedure is similar for $i=7$. For $i=3$, it is necessary to make use of the Peiffer identity of $\mu$, \eqref{action_Lb_1b}, the definition of the action of $\lb{m}$ on $\lb{n}$ and $\lb{q}$ via $\eta$ and \eqref{comp_of_actions_Lb_1}:
				\begin{align*}
				[n,\xi_1(m',q)] & = [n,\mu \xi_1(m',q)] = [n,[m',q]] = [n,[\eta(m'),q]] \\
				& = [[n,\eta(m')],q] - [[n,q],\eta(m')] = [[n,m'],q] - [[n,q],m'].
				\end{align*}
				The conditions required in order to prove the identity for $i=4$ are the same used for $i=3$ except \eqref{comp_of_actions_Lb_1}, which is replaced by \eqref{comp_of_actions_Lb_2}.
				
				Finally, for $i=6$, due to \eqref{action_Lb_4b} and the definition of the action of $\lb{m}$ on $\lb{n}$ via $\eta$, we know that
				\begin{equation*}
				\xi_1([m,m'],q) = [\xi_1(m,q),m'] - [m,\xi_2(q,m')] = [\xi_1(m,q),m'] - [\eta(m),\xi_2(q,m')],
				\end{equation*}
				\noindent but applying \eqref{action_Lb_6b}, we get
				\begin{equation*}
				\xi_1([m,m'],q) = [\xi_1(m,q),m'] + [\eta(m),\xi_1(m',q)] = [\xi_1(m,q),m'] + [m,\xi_1(m',q)],
				\end{equation*}
				\noindent so $(6) = (6') - (6'')$ and the third identity holds. Note that \eqref{action_Lb_6a} and \eqref{action_Lb_6b} are necessary in order to check the fourth and fifth identities respectively.
				
				Checking that $(\mu,\eta)$ is indeed a Leibniz homomorphism follows directly from the definition of the action of $\lb{m}$ on $\lb{n}$ via $\eta$ together with the conditions \eqref{p_equivariant_Lb_1} and \eqref{p_equivariant_Lb_2}. Regarding the equivariance of $(\mu,\eta)$, given $(n,m)\in \lb{n} \rtimes \lb{m}$ and $(q,p) \in \lb{q} \rtimes \lb{p}$,
				\begin{align*}
				(\mu,\eta) ([(q,p),(n,m)]) & = (\mu,\eta) ([q,n] + [p,n] + \xi_2(q,m), [p,m])\\
				& = (\mu ([q,n]) + \mu([p,n]) + \mu\xi_2(q,m), \eta([p,m]))\\
				& = ([q,\mu (n)] + [p, \mu(n)] + [q,m],[p,\eta(m)])\\
				& = ([q,\mu (n)] + [p, \mu(n)] + [q,\eta(m)],[p,\eta(m)])\\
				& = [(q,p),(\mu(n),\eta(m))],
				\end{align*}
				\noindent due to the equivariance of $\mu$ and $\eta$, \eqref{p_equivariant_Lb_1}, \eqref{action_Lb_1a} and the definition of the action of $\lb{m}$ on $\lb{q}$ via $\eta$. Similarly, but using \eqref{p_equivariant_Lb_2} and \eqref{action_Lb_1b} instead of \eqref{p_equivariant_Lb_1} and \eqref{action_Lb_1a}, it can be proved that $(\mu,\eta) ([(n,m),(q,p)]) = [(\mu(n),\eta(m)),(q,p)]$.
				
				The Peiffer identity of $(\mu,\eta)$ follows easily from the homonymous property of $\mu$ and $\eta$, the definition of the action of $\lb{m}$ on $\lb{n}$ via $\eta$ and the conditions \eqref{action_Lb_2a} and \eqref{action_Lb_2b}.
			\end{proof}
			\begin{definition}
				The Leibniz crossed module $(\lb{n}\rtimes \lb{m}, \lb{q}\rtimes \lb{p},(\mu, \eta))$ is called the
				\emph{semidirect product} of the Leibniz crossed modules $(\lb{n},\lb{q},\mu)$ and $(\lb{m},\lb{p},\eta)$.
			\end{definition}
			Note that the semidirect product determines an obvious split extension of $(\lb{m},\lb{p},\eta)$ by $(\lb{n},\lb{q},\mu)$
			\[
			\xymatrix {
			(0,0,0) \ar[r] & (\lb{n},\lb{q},\mu) \ar[r] & (\lb{n}\rtimes \lb{m}, \lb{q}\rtimes \lb{p},(\mu, \eta)) \ar@<0.7ex>[r] & (\lb{m},\lb{p},\eta) \ar[r] \ar@<0.7ex>[l] & (0,0,0)
		    }
			\]
			Conversely, any split extension of $(\lb{m},\lb{p},\eta)$ by $(\lb{n},\lb{q},\mu)$ is isomorphic to
			their semidirect product, where the action of $(\lb{m},\lb{p},\eta)$ on $(\lb{n},\lb{q},\mu)$ is induced by the splitting homomorphism.

			\begin{remark}
				If $(\lb{m},\lb{p},\eta)$ and $(\lb{n},\lb{q},\mu)$ are Leibniz crossed modules and at least one of the following conditions holds,
				\begin{enumerate}
					\item $\Ann(\lb{n})=0=\Ann(\lb{q})$,
					\item $\Ann(\lb{n})=0$ and $[\lb{q},\lb{q}] = \lb{q}$,
					\item $[\lb{n},\lb{n}] = \lb{n}$ and $[\lb{q},\lb{q}] = \lb{q}$,
				\end{enumerate}
				\noindent an action of the crossed module $(\lb{m},\lb{p},\eta)$ on $(\lb{n},\lb{q},\mu)$ can be also defined as a homomorphism of Leibniz crossed modules from $(\lb{m},\lb{p},\eta)$ to $\ol{\Act}(\lb{n},\lb{q},\mu)$. In other words, under one of those conditions, $\ol{\Act}(\lb{n},\lb{q},\mu)$ is the actor of $(\lb{n},\lb{q},\mu)$ and it can be denoted simply by $\Act(\lb{n},\lb{q},\mu)$.
		\end{remark}
			\begin{example}\hfill
				
				\noindent (i) Let $\lb{n}$ be an ideal of a Leibniz algebra $\lb{q}$ and consider the crossed module $(\lb{n},\lb{q},\iota)$, where $\iota$ is the inclusion.  It is easy to check that $\ol{\Act}(\lb{n},\lb{q},\iota) = (X,Y,\iota)$, where $X$ is a Leibniz algebra isomorphic to  $\{ (d,D) \in \Bider(\lb{q}) \, | \, d(q), D(q) \in \lb{n} \text{ for all } q \in \lb{q} \}$ and $Y$ is a Leibniz algebra isomorphic to  $\{ (d,D) \in \Bider(\lb{q}) \, | \, (d_{|\lb{n}},D_{|\lb{n}}) \in \Bider(\lb{n}) \}$.
				
				\noindent (ii) Given a Leibniz algebra $\lb{q}$, it can be regarded as a Leibniz crossed module in two obvious ways, $(0,\lb{q},0)$ and $(\lb{q},\lb{q},\id_{\lb{q}})$. As a particular case of the previous example, one can easily check that $\ol{\Act}(0,\lb{q},0)\cong (0,\Bider(\lb{q}),0)$ and $\ol{\Act}(\lb{q},\lb{q},\id_{\lb{q}})\cong (\Bider(\lb{q}),\Bider(\lb{q}),\id)$.
				
				\noindent (iii) Every Lie crossed module $(\lie{n},\lie{q},\mu)$ can be regarded as a Leibniz crossed module (see for instance \cite[Remark~3.9]{CaKhLa}). Note that in this situation, both the multiplication and the action are antisymmetric. The actor of $(\lie{n},\lie{q},\mu)$ is $(\Der(\lie{q},\lie{n}), \Der(\lie{n},\lie{q},\mu), \Delta)$, where $\Der(\lie{q},\lie{n})$ is the Lie algebra of all derivations from $\lie{q}$ to $\lie{n}$ and $\Der(\lie{n},\lie{q},\mu)$ is the Lie algebra of derivations of the crossed module $(\lie{n},\lie{q},\mu)$ (see \cite{CaLa} for the details). Given $(d,D) \in \Bider(\lie{q},\lie{n})$, both $d$ and $D$ are elements in $\Der(\lie{q},\lie{n})$. Additionally, if we assume that at least one of the conditions from the previous lemma holds, then either $\Ann(\lie{n})=0$ or $[\lie{q},\lie{q}]=\lie{q}$. In this situation, one can easily derive from \eqref{axiom_bider_qn_3} that $\Bider(\lie{q},\lie{n})=\{(d,d) \ | \ d \in \Der(\lie{q},\lie{n})\}$. Besides, the bracket in $\Bider(\lie{q},\lie{n})$ becomes antisymmetric and, as a Lie algebra, it is isomorphic to $\Der(\lie{q},\lie{n})$. Similarly, $\Bider(\lie{n},\lie{q},\mu)$ is a Lie algebra isomorphic to $\Der(\lie{n},\lie{q},\mu)$ and $\ol{\Act}(\lie{n},\lie{q},\mu)$ is a Lie crossed module isomorphic to $\Act(\lie{n},\lie{q},\mu)$.
			\end{example}

\section{Center of a Leibniz crossed module}\label{section5}

Let us assume in this section that $(\lb{n},\lb{q},\mu)$ is a Leibniz crossed module that satisfies at least one of the conditions \eqref{condition1}--\eqref{condition3}. Denote by $\Z(\lb{q})$ the center of the Leibniz algebra $\lb{q}$, which in this case coincides with its annihilator (note that the center and the annihilator are not the same object in general). Consider the canonical homomorphism $(\vp,\psi)$  from $(\lb{n},\lb{q},\mu)$ to $\Act(\lb{n},\lb{q},\mu)$, as in Example \ref{canonical_morph}. It is easy to check that
\begin{align*}
	\Ker(\vp) & = \lb{n}^{\lb{q}} \quad \text{and} \quad	\Ker(\psi) = \st_{\lb{q}}({\lb{n}}) \cap \Z(\lb{q}),
\end{align*}
where $\lb{n}^{\lb{q}} = \{n \in \lb{n} \, | \, [q,n] = [n,q] = 0, \ \text{for all} \ q \in \lb{q}\}$ and $\st_{\lb{q}}({\lb{n}})  =  \{q \in \lb{q} \, | \, [q,n] = [n,q] = 0, \ \text{for all} \ n \in \lb{n}\}$.
Therefore, the kernel of $(\vp,\psi)$ is the Leibniz crossed module $(\lb{n}^{\lb{q}}, \st_{\lb{q}}({\lb{n}}) \cap \Z(\lb{q}),\mu)$. Thus, the kernel of $(\vp,\psi)$ coincides with the center of the crossed module $(\lb{n},\lb{q},\mu)$, as defined in the preliminary version of \cite[Definition 27]{AsCeUs} for crossed modules in modified categories of interest. This definition of center agrees with the categorical notion of center by Huq \cite{Huq} and confirms that our construction of the actor for a Leibniz crossed module is consistent.

\begin{example} Consider the crossed module $(\lb{n},\lb{q},\iota)$, where $\lb{n}$ is an ideal of $\lb{q}$ and $\iota$ is the inclusion. Then, its center is given by the Leibniz crossed module $(\lb{n} \cap \Z(\lb{q}), \Z(\lb{q}), \iota)$. In particular, the center of $(0,\lb{q},0)$ is $(0, \Z(\lb{q}), 0)$ and the center of $(\lb{q},\lb{q},\id_{\lb{q}})$ is $(\Z(\lb{q}), \Z(\lb{q}), \id)$.
\end{example}

By analogy to the definitions given for crossed modules of Lie algebras (see \cite{CaLa}), we can define the crossed module of \emph{inner biderivations} of $(\lb{n},\lb{q},\mu)$, denoted by $\Inn(\lb{n},\lb{q},\mu)$, as $\Im(\vp,\psi)$, which is obviously an ideal. The crossed module of \emph{outer biderivations}, denoted by $\Out(\lb{n},\lb{q},\mu)$, is the quotient of $\Act(\lb{n},\lb{q},\mu)$ by $\Inn(\lb{n},\lb{q},\mu)$.

Let
\[
\xymatrix {
(0,0,0) \ar[r] & (\lb{n},\lb{q},\mu) \ar[r] & (\lb{n}',\lb{q}',\mu') \ar[r] & (\lb{n}'',\lb{q}'',\mu'') \ar[r] & (0,0,0)
}
\]
\noindent be a short exact sequence of crossed modules of Leibniz algebras. Then, there exists a homomorphism of Leibniz crossed modules $(\al, \be) \colon (\lb{n}',\lb{q}',\mu') \to \Act(\lb{n},\lb{q},\mu)$ so that the following diagram is commutative:
\[
\xymatrix@C=1cm{
	(0,0,0) \ar[r] & (\lb{n},\lb{q},\mu) \ar[r] \ar[d] & (\lb{n}',\lb{q}',\mu') \ar[r] \ar[d]^{(\al,\be)} & (\lb{n}'',\lb{q}'',\mu'') \ar[r] \ar[d] & (0,0,0) \\
	(0,0,0) \ar[r] & \Inn(\lb{n},\lb{q},\mu) \ar[r] & \Act(\lb{n},\lb{q},\mu) \ar[r] & \Out(\lb{n},\lb{q},\mu) \ar[r] & (0,0,0)
}
\]
\noindent where $(\al, \be)$ is defined as $\al(n') = (d_{n'},D_{n'})$ and $\be(q') = ((\sigma^{q'}_1,\theta^{q'}_1), (\sigma^{q'}_2,\theta^{q'}_2))$, with
\begin{equation*}
d_{n'} (q) = - [q,n'], \qquad D_{n'} (q) = [n',q],
\end{equation*}
\noindent and
\begin{align*}
\sigma^{q'}_1(n) & = - [n,q'],  \qquad \theta^{q'}_1(n) = [q',n], \\
\sigma^{q'}_2(q) & = - [q,q'],  \qquad \theta^{q'}_2(q) = [q',q],
\end{align*}
\noindent for all $n' \in \lb{n}'$, $q' \in \lb{q}'$, $n \in \lb{n}$, $q \in \lb{q}$.


\end{document}